\documentclass[12pt]{amsart}
\usepackage{amscd,amsmath,amsthm,amssymb}
\usepackage[left]{lineno}
\usepackage{pstcol,pst-plot,pst-3d}
\usepackage{color}
\usepackage{pstricks}
\usepackage{stmaryrd}
\usepackage[utf8]{inputenc}
\usepackage{pstricks-add}
\usepackage[dvips]{graphicx}
\usepackage{epstopdf}
\usepackage{graphicx}
\newpsstyle{fatline}{linewidth=1.5pt}
\newpsstyle{fyp}{fillstyle=solid,fillcolor=verylight}
\definecolor{verylight}{gray}{0.97}
\definecolor{light}{gray}{0.9}
\definecolor{medium}{gray}{0.85}
\definecolor{dark}{gray}{0.6}

 \def\G{{\mathcal G}}

  \def\Mc{{\mathcal M}}

 \def\opn#1#2{\def#1{\operatorname{#2}}} 
 \opn\chara{char} \opn\length{\ell} \opn\pd{pd} \opn\rk{rk}
 \opn\projdim{proj\,dim} \opn\injdim{inj\,dim} \opn\rank{rank}
 \opn\depth{depth} \opn\grade{grade} \opn\height{height}
 \opn\embdim{emb\,dim} \opn\codim{codim}
 
 \opn\Tr{Tr} \opn\bigrank{big\,rank}
 \opn\superheight{superheight}\opn\lcm{lcm}
 \opn\trdeg{tr\,deg}
 \opn\reg{reg} \opn\lreg{lreg} \opn\ini{in} \opn\lpd{lpd}
 \opn\size{size} \opn\sdepth{sdepth}
 \opn\link{link}\opn\fdepth{fdepth}\opn\lex{lex}
 \opn\tr{tr}
 \opn\type{type}
 \opn\gap{gap}
 \opn\arithdeg{arith-deg}
 \opn\div{div} \opn\Div{Div} \opn\cl{cl} \opn\Cl{Cl}
 \opn\Spec{Spec} \opn\Supp{Supp} \opn\supp{supp} \opn\Sing{Sing}
 \opn\Ass{Ass} \opn\Min{Min}\opn\Mon{Mon}
 \opn\Ann{Ann} \opn\Rad{Rad} \opn\Soc{Soc}
 \opn\Im{Im} \opn\Ker{Ker} \opn\Coker{Coker} \opn\Am{Am}
 \opn\Hom{Hom} \opn\Tor{Tor} \opn\Ext{Ext} \opn\End{End}
 \opn\Aut{Aut} \opn\id{id}
 
 \opn\nat{nat}
 \opn\pff{pf}
 \opn\Pf{Pf} \opn\GL{GL} \opn\SL{SL} \opn\mod{mod} \opn\ord{ord}
 \opn\Gin{Gin} \opn\Hilb{Hilb}\opn\sort{sort}
 \opn\PF{PF}\opn\Ap{Ap}
 \opn\mult{mult}
 \opn\bight{bight}
 \opn\aff{aff}
 \opn\relint{relint} \opn\st{st}
 \opn\lk{lk} \opn\cn{cn} \opn\core{core} \opn\vol{vol}  \opn\inp{inp} \opn\nilpot{nilpot}
 \opn\link{link} \opn\star{star}\opn\lex{lex}\opn\set{set}
 \opn\width{wd}
 \opn\Fr{F}
 \opn\QF{QF}
 \opn\G{G}
 \opn\type{type}\opn\res{res}
 \opn\conv{conv}
 \opn\gr{gr}
 
 \def\pot#1#2{#1[\kern-0.28ex[#2]\kern-0.28ex]}

 \opn\dirlim{\underrightarrow{\lim}}
 \opn\inivlim{\underleftarrow{\lim}}
 \let\iso=\cong

 \let\to=\rightarrow
 
 \def\Implies{\ifmmode\Longrightarrow \else
         \unskip${}\Longrightarrow{}$\ignorespaces\fi}
 \def\implies{\ifmmode\Rightarrow \else
         \unskip${}\Rightarrow{}$\ignorespaces\fi}
 \def\iff{\ifmmode\Longleftrightarrow \else
         \unskip${}\Longleftrightarrow{}$\ignorespaces\fi}

 \let\:=\colon
 \newtheorem{Theorem}{Theorem}[section]
 \newtheorem{Lemma}[Theorem]{Lemma}
 
 \newtheorem{Proposition}[Theorem]{Proposition}

 \newtheorem{Example}[Theorem]{Example}
 
 \newtheorem{Definition}[Theorem]{Definition}

 \let\epsilon\varepsilon
 \let\kappa=\varkappa
 \def\qed{\ifhmode\textqed\fi
       \ifmmode\ifinner\quad\qedsymbol\else\dispqed\fi\fi}
 \def\textqed{\unskip\nobreak\penalty50
        \hskip2em\hbox{}\nobreak\hfil\qedsymbol
        \parfillskip=0pt \finalhyphendemerits=0}
 \def\dispqed{\rlap{\qquad\qedsymbol}}
 
 \opn\dis{dis}
 \def\pnt{{\raise0.5mm\hbox{\large\bf.}}}
 
 \opn\Lex{Lex}

\begin{document}

\title {The edge ideal of a graph and its splitting graphs}

\author {J\"urgen Herzog, Somayeh Moradi and Masoomeh Rahimbeigi}

\address{J\"urgen Herzog, Fachbereich Mathematik, Universit\"at Duisburg-Essen, Campus Essen, 45117
Essen, Germany} \email{juergen.herzog@uni-essen.de}

\address{Somayeh Moradi, Department of Mathematics, School of Science, Ilam University,
P.O.Box 69315-516, Ilam, Iran}
\email{so.moradi@ilam.ac.ir}

\address{Masoomeh Rahimbeigi, Department of Mathematics, University of Kurdistan, Post
Code 66177-15175, Sanandaj, Iran}
\email{rahimbeigi$_{-}$masoome@yahoo.com}

\dedicatory{ }

\begin{abstract}
 We introduce and study the concept which we call the splitting of a graph and compare algebraic properties of the edge ideals of graphs and those of their splitting graphs.
\end{abstract}

\thanks{}




\maketitle

\setcounter{tocdepth}{1}

\section*{Introduction}

For any monomial ideals $I$ and $J$ it is known that $\reg(I+J)\leq \reg(I)+\reg(J)-1$ and $\projdim(I+J)\leq \projdim(I)+\projdim(J)+1$, see \cite{Kal} and \cite{H}. Suppose we are given a finite simple graph $G'$ with connected components $G_1$ and $G_2$ and suppose we identify some vertex of $G_1$ with some vertex of $G_2$ to obtain the graph $G$. Then for the edge ideal, the above inequalities imply
$(i)\ \reg(I(G))\leq \reg(I(G'))$ and
$(ii)\ \projdim (I(G))\leq \projdim(I(G'))$.
The graph $G'$ may also be considered as a splitting graph of $G$ in the following sense. For a finite simple graph $G$, let $V(G)$ and $E(G)$ denote the vertex set and the edge set of $G$, respectively.
We call a graph $G'$ a splitting graph of $G$, if there exists a surjective map $\alpha\: V(G')\to V(G)$ such that $\alpha(e):=\{\alpha(v),\alpha(w)\}$ is an edge of $G$ for all edges $e = \{v,w\}$ of $G'$, and such that the map $E(G')\to E(G)$, $e\mapsto \alpha(e)$ is bijective.

This kind of splitting graphs naturally occur as graphs whose edge ideals are obtained by applying Kalai's shifting operator.

In this paper we study the question of whether the above inequalities $(i)$ and $(ii)$ are valid for any splitting graph of $G$. It turns out that this problem is harder than expected. In Theorem~\ref{special1}, we succeed to prove the desired inequalities for special classes of splittings.

On the other hand there are big classes of graphs for which the inequalities $(i)$ and $(ii)$ hold.
We show in Proposition~\ref{truepd} that the inequality $(i)$ holds if $G$ is a sequentially Cohen-Macaulay graph and in Proposition~\ref{bound}, it is proved that the inequality $(ii)$ holds when $G$ is
a chordal graph, a weakly chordal graph, a sequentially Cohen-Macaulay bipartite graph, an unmixed bipartite graph, a very well-covered graph or a $C_5$-free vertex decomposable graph.

In the literature, there is a well-studied concept of splittable monomial ideals due to Eliahou-Kervaire \cite{EK}. For this kind of splitting, the graded Betti numbers of a splittable monomial ideal $I=J+K$ can be expressed in terms of the graded Betti numbers of $J$, $K$ and $J\cap K$. Simple examples show that there is in general no comparison  possible for the graded Betti numbers of the edge ideal of a graph and its splitting graph. This is one of the reasons why it is hard to prove $(i)$ and $(ii)$ in general. However, we expect that $\beta_i(I(G))\leq \beta_i(I(G'))$ for all $i$ and we can prove this for special splittings.
At the end of the paper, we briefly discuss the relationship between shifted graphs and splitting graphs.

\section{Splitting graphs}
\label{secthree}

In this section we introduce the concept of splitting graphs of a given graph and compare their algebraic properties.

\begin{Definition}
{\em Let $G$ be a finite simple graph. We say that the graph $G'$ is a {\em splitting graph} of $G$, if there exists
a surjective map $\alpha\: V(G')\to V(G)$ such that $\alpha(e):=\{\alpha(v),\alpha(w)\}\in E(G)$ for all $e=\{v,w\}\in E(G')$ and such that the map $E(G')\to E(G)$, $e\mapsto \alpha(e)$ is bijective. We call $\alpha$ a splitting map of $G$.}
\end{Definition}

Observe that if the edges $e,f\in E(G')$ are neighbors in $G'$, then the edges $\alpha(e)$ and $\alpha(f)$ are neighbors in $G$.

Figure~\ref{one} shows an example of a splitting graph of a graph $G$.

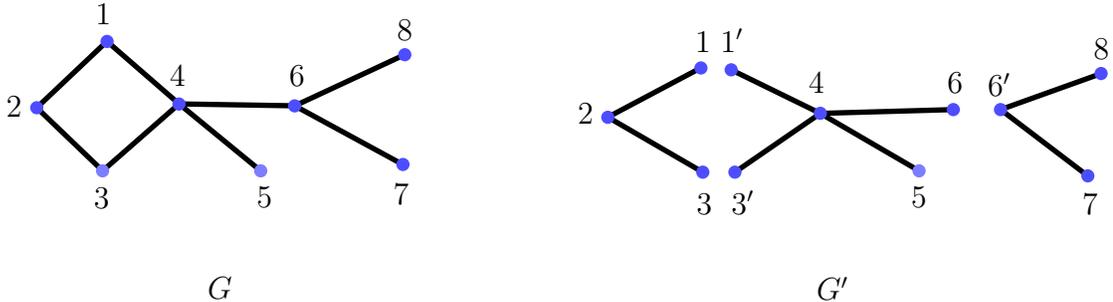
\begin{figure}[hbt]
\begin{center}
\newrgbcolor{ududff}{0.30196078431372547 0.30196078431372547 1.}
\newrgbcolor{xdxdff}{0.49019607843137253 0.49019607843137253 1.}
\psset{xunit=1.0cm,yunit=1.0cm,algebraic=true,dimen=middle,dotstyle=o,dotsize=5pt 0,linewidth=1.6pt,arrowsize=3pt 2,arrowinset=0.25}
\begin{pspicture*}(-10.28851230389455,-2.995971620646762)(10.315373556715628,5.406866850612861)
\psline[linewidth=2.pt](-1.2835852737024576,0.7135818528296999)(-0.047614897009425144,1.369402869034166)
\psline[linewidth=2.pt](-1.2835852737024576,0.7135818528296999)(-0.02239101177079183,-0.017910819090666306)
\psline[linewidth=2.pt](0.3559672668087079,1.3441789837955327)(1.5414898730244735,0.7640296233069664)
\psline[linewidth=2.pt](1.5414898730244735,0.7640296233069664)(0.4064150372859745,-0.017910819090666306)
\psline[linewidth=2.pt](1.5414898730244735,0.7640296233069664)(3.3071618397288054,0.8144773937842331)
\psline[linewidth=2.pt](1.5414898730244735,0.7640296233069664)(2.853131905433406,0.)
\psline[linewidth=2.pt](3.937758970694638,0.8144773937842331)(5.2746248883422036,1.2937312133182661)
\psline[linewidth=2.pt](3.937758970694638,0.8144773937842331)(5.09805769167177,-0.06835858956793293)
\rput[tl](-0.12,1.8595522295227323){$1$}
\rput[tl](-1.680629784751236,0.907312133182661){$2$}
\rput[tl](-0.104629864157125,-0.29537355671563276){$3$}
\rput[tl](1.3849226763540403,1.3074524750066){$4$}
\rput[tl](2.7508930530470726,-0.205974419542661){$5$}
\rput[tl](3.230594643058372,1.3098506395114328){$6$}
\rput[tl](5.02238603595587,-0.30452124315327){$7$}
\rput[tl](5.172833806433137,1.7677611476136659){$8$}
\rput[tl](0.21462395537690797,1.8895522295227323){$1'$}
\rput[tl](0.358206119195041,-0.26582132719289943){$3'$}
\rput[tl](3.765967888785572,1.305074524750066){$6'$}
\psline[linewidth=2.pt](-7.942690976701652,1.7225372623750326)(-8.875974730531084,0.8397012790228664)
\psline[linewidth=2.pt](-8.875974730531084,0.8397012790228664)(-8.,0.)
\psline[linewidth=2.pt](-8.,0.)(-6.984183337633586,0.8901490495001331)
\psline[linewidth=2.pt](-6.984183337633586,0.8901490495001331)(-7.942690976701652,1.7225372623750326)
\psline[linewidth=2.pt](-6.984183337633586,0.8901490495001331)(-5.445526338076954,0.8649251642614998)
\psline[linewidth=2.pt](-6.984183337633586,0.8901490495001331)(-5.899556272372354,0.)
\psline[linewidth=2.pt](-5.445526338076954,0.8649251642614998)(-3.982540994236222,1.5459700657045994)
\psline[linewidth=2.pt](-5.445526338076954,0.8649251642614998)(-4.0077648794748555,0.08298472186386696)
\rput[tl](-9.27900205781017,0.99402869034166){$2$}
\rput[tl](-8.1153714326619,-0.215974419542661){$3$}
\rput[tl](-7.105526649065386,1.405074524750066){$4$}
\rput[tl](-5.95134735428142,-0.20582132719289943){$5$}
\rput[tl](-5.52093534747387,1.402984099886993){$6$}
\rput[tl](-4.133884305668022,-0.1710452124315327){$7$}
\rput[tl](-4.084332076145289,2.017910819090654){$8$}
\rput[tl](-8.094034288133452,2.2288060490567656){$1$}
\rput[tl](-6.605825059054086,-1.4052245072154987){$G$}
\rput[tl](1.491042102547207,-1.4052245072154987){$G'$}
\begin{scriptsize}
\psdots[dotstyle=*,linecolor=ududff](-1.2835852737024576,0.7135818528296999)
\psdots[dotstyle=*,linecolor=ududff](-0.047614897009425144,1.369402869034166)
\psdots[dotstyle=*,linecolor=ududff](-0.02239101177079183,-0.017910819090666306)
\psdots[dotstyle=*,linecolor=ududff](0.3559672668087079,1.3441789837955327)
\psdots[dotstyle=*,linecolor=ududff](1.5414898730244735,0.7640296233069664)
\psdots[dotstyle=*,linecolor=ududff](0.4064150372859745,-0.017910819090666306)
\psdots[dotstyle=*,linecolor=ududff](3.3071618397288054,0.8144773937842331)
\psdots[dotstyle=*,linecolor=xdxdff](2.853131905433406,0.)
\psdots[dotstyle=*,linecolor=ududff](3.937758970694638,0.8144773937842331)
\psdots[dotstyle=*,linecolor=ududff](5.2746248883422036,1.2937312133182661)
\psdots[dotstyle=*,linecolor=ududff](5.09805769167177,-0.06835858956793293)
\psdots[dotstyle=*,linecolor=ududff](-7.942690976701652,1.7225372623750326)
\psdots[dotstyle=*,linecolor=ududff](-8.875974730531084,0.8397012790228664)
\psdots[dotstyle=*,linecolor=xdxdff](-8.,0.)
\psdots[dotstyle=*,linecolor=ududff](-6.984183337633586,0.8901490495001331)
\psdots[dotstyle=*,linecolor=ududff](-5.445526338076954,0.8649251642614998)
\psdots[dotstyle=*,linecolor=xdxdff](-5.899556272372354,0.)
\psdots[dotstyle=*,linecolor=ududff](-3.982540994236222,1.5459700657045994)
\psdots[dotstyle=*,linecolor=ududff](-4.0077648794748555,0.08298472186386696)
\end{scriptsize}
\end{pspicture*}
\end{center}
\caption{A graph $G$ and a splitting graph $G'$ of $G$.}
\label{one}
\end{figure}

In the example of Figure~\ref{one},  we define $\alpha\: V(G')\to V(G)$ by $\alpha(i)=i$  for $i=1,\ldots,8$ and $\alpha(1')=1$, $\alpha(2')=2$ and $\alpha(3') =3$. With respect to $\alpha$, $G'$ is indeed a splitting graph of $G$.

\medskip
A splitting graph does not necessarily need to decompose a graph into several connected components, as the example in Figure~\ref{two} shows. The graph $G''$ illustrated in Figure~\ref{two} is another splitting graph of the graph $G$ depicted in Figure~\ref{one}. This splitting graph is indecomposable.

\begin{figure}[hbt]
\begin{center}
\newrgbcolor{ududff}{0.30196078431372547 0.30196078431372547 1.}
\psset{xunit=0.8cm,yunit=0.8cm,algebraic=true,dimen=middle,dotstyle=o,dotsize=5pt 0,linewidth=1.6pt,arrowsize=3pt 2,arrowinset=0.25}
\begin{pspicture*}(-4.,0.7)(6.,6.)
\psline[linewidth=2.pt](-0.5,4.14)(0.12,3.04)
\psline[linewidth=2.pt](-1.44,3.04)(0.12,3.04)
\psline[linewidth=2.pt](0.12,3.04)(0.62,2.08)
\psline[linewidth=2.pt](0.12,3.04)(1.64,3.)
\psline[linewidth=2.pt](1.64,3.)(3.,4.)
\psline[linewidth=2.pt](1.64,3.)(3.,2.)
\rput[tl](0.06,1.18){$\large{G''}$}
\psline[linewidth=2.pt](-1.44,3.04)(-2.1,4.1)
\psline[linewidth=2.pt](-2.1,4.1)(-1.32,5.)
\begin{scriptsize}
\psdots[dotstyle=*,linecolor=ududff](-0.5,4.14)
\psdots[dotstyle=*,linecolor=ududff](-1.44,3.04)
\psdots[dotstyle=*,linecolor=ududff](0.12,3.04)
\psdots[dotstyle=*,linecolor=ududff](0.62,2.08)
\psdots[dotstyle=*,linecolor=ududff](3.,4.)
\psdots[dotstyle=*,linecolor=ududff](3.,2.)
\psdots[dotstyle=*,linecolor=ududff](1.64,3.)
\psdots[dotstyle=*,linecolor=ududff](-2.1,4.1)
\psdots[dotstyle=*,linecolor=ududff](-1.32,5.)
\end{scriptsize}
\end{pspicture*}
\end{center}
\caption{An indecomposable  splitting graph  of $G$.}
\label{two}
\end{figure}
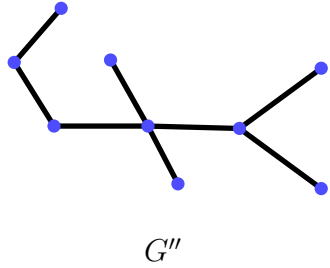

We expect the following properties to hold. Let $G'$ be a splitting graph of $G$. Then
\begin{itemize}
\item[(i)] $\projdim (I(G))\leq \projdim (I(G'))$;
\item[(ii)] $\reg (I(G))\leq \reg (I(G'))$;
\item[(iii)] $\beta_i (I(G))\leq \beta_i(I(G'))$ for all $i$.
\end{itemize}
For the graded Betti numbers,  an inequality as (iii) is not valid. Indeed, let $G$ be the path  graph with edges $E(G)=\{\{1,2\},\{2,3\}\}$ and $G'$ the splitting graph  of $G$ with edges $E(G')=\{\{1,2\},\{3,4\}\}$. Then $\beta_{1,2}(I(G))=1$ and  $\beta_{1,2}(I(G'))=0$, while $\beta_{1,3}(I(G))=0$ and $\beta_{1,3}(I(G'))=1$.

\medskip
Throughout this paper $G'$ denotes a splitting graph  of $G$, and  $S$ and  $S'$ are the polynomial rings over a given field $K$ in the variables corresponding to $V(G)$ and $V(G')$, respectively. Related to  the above inequalities  one may also expect that
\begin{itemize}
\item[(iv)] $\dim (S'/I(G'))\geq \dim (S/I(G))$;
\item[(v)] $\depth (S'/I(G'))\geq \depth  (S/I(G))$.
\end{itemize}

At present we are not able to prove (i), (ii) and (iii) in full generality. For splitting graphs which are special in the sense of Definition~\ref{special}, (i) and (ii) can be shown. Also (iii) can be proved for splitting graphs satisfying condition (2) of Definition  \ref{special}. In Proposition~\ref{trueornottrue}, it is shown that (iv) holds for any graph $G$ and any splitting graph of $G$ and (v) holds
for path graphs and cycle graphs of even length.

For a vertex $v$ of the graph $G$, let $N_{G}(v)=\{u\in V(G):\ \{u,v\}\in E(G)\}$ and $N_{G}[v]=N_{G}(v)\cup\{v\}$.
\begin{Definition}\label{special}
{\em  A splitting map $\alpha:V(G')\rightarrow V(G)$ is called  {\em special}, if either

(1) for any two vertices $v,v'\in V(G')$ with $\alpha(v)=\alpha(v')$, any vertex in $N_{G'}(v)$ is adjacent to any vertex in $N_{G'}(v')$; or

(2) for any two vertices $v,v'\in V(G')$ with $\alpha(v)=\alpha(v')$, $v$ and $v'$ belong to different connected components of $G'$.

Also $G'$ is called a special splitting graph of $G$ if the corresponding splitting map is special.
}
\end{Definition}

\begin{Theorem}
\label{special1}
Let $G'$ be a special splitting graph of $G$. Then
\begin{itemize}
\item[(i)] $\projdim (I(G))\leq \projdim (I(G'))$;
\item[(ii)] $\reg (I(G))\leq \reg (I(G'))$;
\end{itemize}
\end{Theorem}

For the proof of Theorem~\ref{special1} we need the following result.

\begin{Lemma}\label{kernel}
Let $G$ be a graph. Let $x,y\in V(G)$ such that $N_G[x]\cap N_G[y]=\emptyset$. Then $I:(x-y)=I+(zw: \ z\in N_G(x), w\in N_G(y))$.
\end{Lemma}
\begin{proof}
It is obvious that the right hand side is contained in the left hand side. To prove the other inclusion,  let $f\in (I:x-y)$ be a polynomial. Then
we can write $f=\sum_{r=1}^m \lambda_i u_i$ such that $\lambda_i\in K$ and $u_i$'s are pairwise distinct monomials in $\Supp(f)$ and non of them belong to $I$. Then

\begin{eqnarray}
\label{monomial}
f(x-y)=(\sum_{r=1}^m \lambda_i u_i)x-(\sum_{r=1}^m \lambda_i u_i)y\in I.
\end{eqnarray}

 We claim that for all $i$,  $u_ix\in I$ and $u_iy\in I$. By contradiction, assume that there exists $i$ such that $u_ix\notin I$  and set $A=\{u_i\in \Supp(f):\ u_i x\notin I\}$. Let $u_j\in A$ be a monomial which has the greatest degree in $x$ among the elements of $A$ and without loss of generality let $j=1$. Let $u_1=x^ay^b w$ for some monomial $w$ which is divided by neither $x$ nor $y$. Since $xu_1\notin I$, and $I$ is a monomial ideal, by (\ref{monomial}), we should have $\lambda_1xu_1=\lambda_{\ell}yu_{\ell}$ for some $\ell$. Then $u_{\ell}=x^{a+1}y^{b-1} w$. So by our assumption on $u_1$, we have  $u_{\ell}\notin A$ and then $xu_{\ell}=x^{a+2}y^{b-1}w\in I$. So $xyu_{\ell}\in I$. Therefore, $\lambda_1x^2u_1=\lambda_{\ell}xyu_{\ell}\in I$. Since $I$ is a squarefree monomial ideal, $xu_1\in I$, a contradiction. So
 we have  $u_ix\in I$ for any $i$. By similar argument $u_iy\in I$ for any $i$. This means that there exists $z\in N_G(x)$ such that $z$ divides $u_i$ and there exists
 $w\in N_G(y)$ such that $y$ divides $u_i$. Thus $u_i\in (zw:\ \ z\in N_G(x), w\in N_G(y))$ for any $i$.

\end{proof}

\begin{proof}[Proof of Theorem~\ref{special1}]
$(i)$ Assume that $G'$ a special splitting graph of $G$ with the splitting map $\alpha$ satisfying condition (1) of Definition~\ref{special}.
We set $G_0=G'$. Fix two vertices $x,y\in V(G')$ such that $\alpha(x)=\alpha(y)$ and let $G_1$ be a graph with the vertex set $V(G_1)=V(G_0)\setminus \{y\}$ and the edge set $E(G_1)=(E(G_0)\setminus \{\{y,w\}:\ w\in N_{G_0}(y)\})\cup\{\{x,w\}:\ w\in N_{G_0}(y)\}$. Then considering the map $\alpha_0: V(G_0)\rightarrow V(G_1)$ with

\begin{eqnarray}
\label{alpha}
\alpha_0(v) = \left\{
\begin{array}{ll}
v,  &   \text{if $v\neq y$};\\
x, &  \text{if $v=y$}\\
\end{array}\right.
\end{eqnarray}

$G_0$ is a special splitting graph of $G_1$. With the same argument as above one can define the sequence of graphs $G'=G_0,G_1,\ldots,G_t=G$ such that $G_{i-1}$ is a special splitting graph of $G_i$  with the splitting map $\alpha_i: V(G_{i-1})\rightarrow V(G_i)$ for any $1 \leq i \leq t$.
So it is enough to show that for such kind of splitting map $\alpha_i$, we have $\projdim (I(G_i))\leq \projdim (I(G_{i-1}))$.
We prove this inequality for $i=1$ and the others can be proved in the same way. Set $I=I(G')$.
Considering the short exact sequence
\begin{eqnarray}
\label{seq1}
0\rightarrow (S'/(I:x-y))(-1)  \rightarrow S'/I \rightarrow  S'/(I,x-y) \rightarrow 0,
\end{eqnarray}
we have
\begin{eqnarray}
\label{exact}
\ \ \ \ \ \  \projdim (S'/(I,x-y))\leq \max\{\projdim (S'/I),\projdim (S'/(I:x-y))+1\}.
\end{eqnarray}

Our assumptions on the splitting map and Lemma~\ref{kernel} imply that $(I:x-y)=I$. Therefore, (\ref{exact}) implies that
\begin{eqnarray}
\label{pd1}
\projdim (S'/(I,x-y))\leq \projdim (S'/I)+1.
\end{eqnarray}

Note that $(I,x-y)=(I(G_1),x-y)$. One can see that $x-y$ is a nonzero-divisor modulo $I(G_1)$.  Indeed,  since $y$ does not appear in the support of the generators of $I(G_1)$, $x-y$ behaves like a new variable.
Thus
\begin{eqnarray}
\label{pd2}
\projdim (S'/(I,x-y))&=&\projdim (S'/(I(G_1),x-y))\\
&=&\projdim  (S'/I(G_1))+1.\nonumber
 \end{eqnarray}
The desired conclusion follows from (\ref{pd1}) and (\ref{pd2}).

Assume that $G'$ a special splitting graph of $G$ with the splitting map $\alpha$ satisfying condition (2) of Definition~\ref{special}.
Let $G'_1,\ldots,G'_r$ be the connected components of $G'$. For any $1\leq  i\leq  r$, let $G_i$ be the graph with the vertex set $\alpha(V(G'_i))$ and the edge set $\alpha(E(G'_i))$. Then $I(G)=\sum_{i=1}^ r I(G_i)$. So by using \cite[Corollary 3.2]{H}, $\projdim (S'/I(G))\leq \sum_{i=1}^ r\projdim (S'/I(G_i))$. Since each $G'_i$ is connected, condition (2) of Definition~\ref{special} implies that $I(G'_i)=I(G_i)$. Therefore
we get  $$\projdim (S'/I(G))\leq \sum_{i=1}^ r\projdim (S'/I(G'_i))=\projdim (S'/I(G')).$$
The last equality follows from the fact that the ideals $I(G'_i)$ live in disjoint sets of variables.

$(ii)$ Let $G'$ be a special splitting graph of $G$ with the splitting map $\alpha$ satisfying condition (1) of Definition~\ref{special}.
With the same notation as in part $(i)$, it is enough to prove that
$\reg (I(G_i))\leq \reg (I(G_{i-1}))$.
We prove this inequality for $i=1$ and the others can be proved in the same way.
Considering again the short exact sequence $(\ref{seq1})$, we have
\begin{eqnarray}
\label{exact2}
\reg (S'/(I,x-y))\leq \max\{\reg (S'/I),\reg (S'/(I:x-y))\}.
\end{eqnarray}

The equality $(I:x-y)=I$ and (\ref{exact2}) imply that
\begin{eqnarray}
\label{pd3}
\reg (S'/(I,x-y))\leq \reg (S'/I).
\end{eqnarray}

As mentioned above, $(I,x-y)=(I(G_1),x-y)$ and $x-y$ is a nonzero-divisor modulo $I(G_1)$.
So
\begin{eqnarray}
\label{pd4}
\reg (S'/(I,x-y))=\reg (S'/(I(G_1),x-y))=\reg  (S'/I(G_1)).
 \end{eqnarray}
The desired conclusion follows from (\ref{pd3}) and (\ref{pd4}).

If $G'$ is a special splitting graph of $G$ with the splitting map $\alpha$ satisfying condition (2) of Definition~\ref{special}, then with the similar argument as part $(i)$ one can get the result.
\end{proof}

A subset $C\subseteq V(G)$ is called a
vertex cover of $G$ if it intersects all edges of $G$ and a
vertex cover of $G$ is called minimal if it has no proper
subset which is also a vertex cover of $G$. We set $\bight(I(G))=\max\{|C|:\ C\  \textrm{is a minimal vertex cover of }\ G\}$.

\begin{Proposition}
\label{truepd}
Let $G$ be a graph for which $\projdim (S/I(G))=\bight(I(G))$. Then $\projdim (I(G))\leq \projdim (I(G'))$. In particular, we have $\projdim (I(G))\leq \projdim (I(G'))$ when $G$ is a sequentially Cohen-Macaulay graph.

\end{Proposition}

\begin{proof}
Let $G$ be a graph with $\projdim (S/I(G))=\bight(I(G))$. By \cite[Corollary 3.33]{MV}, $\bight(I(G'))-1 \leq \projdim (I(G'))$. So it is enough to show that $\bight(I(G))\leq \bight(I(G'))$. Let $C$ be a minimal vertex cover of $G$ with $\bight(I(G))=|C|$ and let $C'$ be the preimage of $C$ under the surjective map $\alpha\: V(G')\to V(G)$ attached to the splitting graph of $G$. Then $C'$ is  a vertex cover of $G'$. Let $D'$ be a minimal vertex cover of $G'$ with $D'\subseteq C'$. One can see that $\alpha(D')$ is a vertex cover of $G$. Also $\alpha(D')\subseteq\alpha(C')=C$. Since $C$ is a minimal vertex cover of $G$, we should have $\alpha(D')=C$. The inequality  $|C|\leq |D'|\leq \bight(I(G'))$ completes the proof.
\end{proof}

\begin{Proposition}
\label{bound}
If $G$ is a graph with $\reg (I(G))=\nu(G)+1$, then $\reg (I(G))\leq \reg (I(G'))$. In particular,  $\reg (I(G))\leq \reg (I(G'))$ in the following cases:
\begin{itemize}
  \item $G$ is a chordal graph;
  \item $G$ is a weakly chordal graph;
  \item $G$ is a sequentially Cohen-Macaulay bipartite graph;
  \item $G$ is an unmixed bipartite graph;
  \item $G$ is a very well-covered graph;
  \item $G$ is a $C_5$-free vertex decomposable graph.
\end{itemize}
\end{Proposition}

\begin{proof}
\label{matching}
Let $G$ be a graph with $\reg (I(G))=\nu(G)+1$. By \cite[Lemma 2.2]{K}, we have $\nu(G')+1\leq \reg (I(G'))$.
Thus to prove our statement we need to show that
\begin{eqnarray}
\label{easier}
\nu(G)\leq \nu(G').
\end{eqnarray}
Let $\alpha\: V(G')\to V(G)$ be the surjective map attached to the splitting graph of $G$. Note that if the edges $e,f\in E(G')$ are neighbors in $G'$, then the edges $\alpha(e)$ and $\alpha(f)$ are neighbors in $G$.
Let $e_1,\ldots,e_r$ be any induced matching of $G$ and for any $1\leq i\leq r$, let $e'_i\in E(G')$ be such that $\alpha(e'_i)=e_i$. Then $e'_1,\ldots,e'_r$ are pairwise disjoint. It is enough to show that $e'_1,\ldots,e'_r$ is an induced matching in $G'$. Suppose that $e'_1,\ldots,e'_r$ is not an induced matching, then there exists an edge $e'\in E(G')$ with neighbors $e'_i$ and $e'_j$ for some distinct $i,j\in\{1,\ldots,r\}$. Then $\alpha(e')$ is also neighbor with $e_i$ and $e_j$ in $G$ for some $i,j\in\{1,\ldots,r\}$, a contradiction.
The last statements follows from \cite[Corollary 6.9]{HT1}, \cite[Theorem 14]{W},  \cite[Theorem 3.3]{VT}, \cite[Theorem 1.1]{Ku}, \cite[Theorem 1.3]{MM} and \cite[Theorem 2.4]{MK}, respectively.

\end{proof}

\begin{Proposition}
\label{totalbetti}
Let $G'$ be a special splitting graph of $G$ with the splitting map $\alpha$ satisfying condition (2) of Definition~\ref{special}.
Then $\beta_i(I(G))\leq \beta_i(I(G'))$.
\end{Proposition}

\begin{proof}
First assume that $G'$ has two connected components $G'_1$ and $G'_2$ and let $G_i$ be the graph with the vertex set $\alpha(V(G'_i))$ and the edge set $\alpha(E(G'_i))$, for $i=1,2$. Then $I(G)=I(G_1)+I(G_2)$. So by \cite[Corollary~3.1]{H}, $\beta_i (S/I(G))\leq \sum_{j=0}^i \beta_j (S/I(G_1))\beta_{i-j} (S/I(G_2))$. Since each $G'_i$ is connected, condition (2) of Definition~\ref{special} implies that $I(G'_i)=I(G_i)$. Therefore
we get  $$\beta_i (S/I(G))\leq \sum_{j=0}^i \beta_j (S/I(G'_1))\beta_{i-j} (S/I(G'_2))=\beta_i (S/I(G')),$$
The last equality follows from the fact that the ideals $I(G'_i)$ live in disjoint sets of variables.
In general if $G'$ has $r$ connected components, then repeating the above argument, one can get the desired inequality.
\end{proof}

\begin{Proposition}
\label{trueornottrue}
The inequality $\dim (S'/I(G'))\geq \dim (S/I(G))$ is valid for any graph and we have $\depth (S'/I(G'))\geq \depth  (S/I(G))$, when $G$ is a path graph or a cycle of even length.
\end{Proposition}

\begin{proof}
Let $G$ be an arbitrary graph, $n=\dim (S)$ and $n'=\dim (S')$. Then $n=|V(G)|$, $n'=|V(G')|$, and we have $\dim (S/I(G))=n-\mu$ and $\dim (S'/I(G'))=n'-\mu'$, where $\mu$ is the cardinality of a vertex cover of $G$ of minimal size, and $\mu'$ is the cardinality of a vertex cover of $G'$ of minimal size. Let $C$ be a vertex cover of $G$ with $|C|=\mu$, and let $C'$ be the preimage of $C$ under the surjective map $\alpha\: V(G')\to V(G)$ attached to the splitting graph of $G$. Then $C'$ is  a vertex cover of $G'$, but not necessarily of minimal size. Moreover, $|C'|\leq |C|+n'-n$. Thus $\mu'\leq \mu+n'-n$, which is equivalent to saying that $n'-\mu'\geq n-\mu$, as desired.

Now, let $G$ be  a path graph.  By the theorem of Auslander-Buchsbaum,  one has $\depth (S/I(G))=n- \projdim (S/I(G))$, and  $\depth (S'/I(G'))=n'- \projdim (S'/I(G'))$. The graph $G'$ has $r$ components which are path graphs $P_{n_1},\ldots,P_{n_r}$ for some $n_1,\ldots,n_r$.
By \cite[Corollary 7.7.35]{Ja}, we have
\begin{eqnarray*}
\label{depthpath}
 \projdim (S'/I(G'))=\sum_{i=1}^r\projdim (S/I(P_{n_i}))&\leq &\sum_{i=1}^{r}\frac{2n_i}{3}
  \\
  &=&\frac{2(n_1+n_2+\cdots+n_r)}{3}\\
  &=&\frac{2(n+r-1)}{3}=\frac{2n+2r-2}{3}.
 \end{eqnarray*}
Moreover, $\projdim(S/I(G))\geq \frac{2n-2}{3}$.
Hence
$\depth (S'/I(G'))-\depth (S/I(G))=r-1-[\projdim (S'/I(G'))-\projdim (S/I(G))]\geq r-1-[\frac {2n+2r-2}{3}-\frac{2n-2}{3}]=r-1 - \frac{2r}{3}=\frac{r}{3}-1>-1$. Hence  $\depth (S'/I(G'))-\depth (S/I(G))\geq 0$.

The argument for  cycles of even length is similar.
\end{proof}

The inequality $\depth (S'/I(G'))\geq \depth  (S/I(G))$ does not hold in general as the following example shows.

\begin{Example}
\label{depth}
Let $G$ and $G'$ be the graphs depicted in Figure \ref{depthnottrue}, where $G'$ is a splitting graph of $G$. Then $\depth  (S/I(G))=3$ and $\depth (S'/I(G'))=2$.

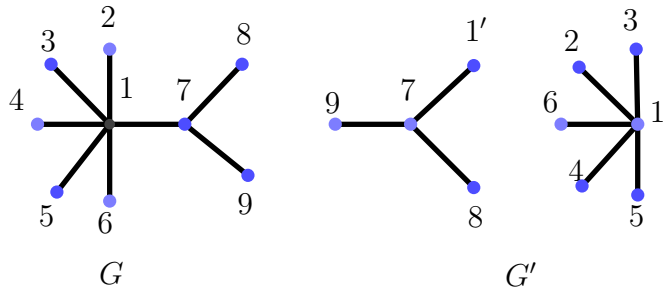
\begin{figure}[hbt]
\begin{center}
\newrgbcolor{ududff}{0.30196078431372547 0.30196078431372547 1.}
\newrgbcolor{xdxdff}{0.49019607843137253 0.49019607843137253 1.}
\psset{xunit=1.0cm,yunit=1.0cm,algebraic=true,dimen=middle,dotstyle=o,dotsize=5pt 0,linewidth=1.6pt,arrowsize=3pt 2,arrowinset=0.25}
\begin{pspicture*}(-4.44,-2.86)(10.56,1.88)
\psline[linewidth=2.pt](0.,0.)(1.,0.)
\psline[linewidth=2.pt](0.,0.)(-0.78,0.8)
\psline[linewidth=2.pt](0.,0.)(-0.96,0.)
\psline[linewidth=2.pt](0.,0.)(-0.7,-0.9)
\psline[linewidth=2.pt](0.,0.)(0.,1.)
\psline[linewidth=2.pt](0.,0.)(0.,-1.02)
\psline[linewidth=2.pt](1.,0.)(1.76,0.8)
\psline[linewidth=2.pt](1.,0.)(1.84,-0.68)
\psline[linewidth=2.pt](4.,0.)(4.84,0.78)
\psline[linewidth=2.pt](4.,0.)(4.84,-0.84)
\psline[linewidth=2.pt](4.,0.)(3.,0.)
\psline[linewidth=2.pt](7.02,0.)(7.,1.)
\psline[linewidth=2.pt](7.02,0.)(6.24,0.76)
\psline[linewidth=2.pt](7.02,0.)(6.,0.)
\psline[linewidth=2.pt](7.02,0.)(6.28,-0.82)
\psline[linewidth=2.pt](7.02,0.)(7.02,-0.94)
\rput[tl](0.12,0.62){$1$}
\rput[tl](-0.12,1.56){$2$}
\rput[tl](-0.92,1.24){$3$}
\rput[tl](-1.34,0.46){$4$}
\rput[tl](-0.94,-1.08){$5$}
\rput[tl](-0.16,-1.18){$6$}
\rput[tl](0.88,0.6){$7$}
\rput[tl](1.68,1.34){$8$}
\rput[tl](1.7,-0.92){$9$}
\rput[tl](7.18,0.3){$1$}
\rput[tl](6.04,1.24){$2$}
\rput[tl](6.82,1.52){$3$}
\rput[tl](6.1,-0.48){$4$}
\rput[tl](6.9,-1.08){$5$}
\rput[tl](5.76,0.44){$6$}
\rput[tl](3.86,0.6){$7$}
\rput[tl](4.76,-1.1){$8$}
\rput[tl](4.72,1.42){$1'$}
\rput[tl](2.86,0.42){$9$}
\rput[tl](-0.14,-1.84){$G$}
\rput[tl](5.26,-1.84){$G'$}
\begin{scriptsize}
\psdots[dotstyle=*,linecolor=ududff](1.,0.)
\psdots[dotstyle=*,linecolor=ududff](-0.78,0.8)
\psdots[dotstyle=*,linecolor=xdxdff](-0.96,0.)
\psdots[dotstyle=*,linecolor=ududff](-0.7,-0.9)
\psdots[dotstyle=*,linecolor=xdxdff](0.,1.)
\psdots[dotstyle=*,linecolor=xdxdff](0.,-1.02)
\psdots[dotstyle=*,linecolor=ududff](1.76,0.8)
\psdots[dotstyle=*,linecolor=ududff](1.84,-0.68)
\psdots[dotstyle=*,linecolor=xdxdff](4.,0.)
\psdots[dotstyle=*,linecolor=ududff](4.84,0.78)
\psdots[dotstyle=*,linecolor=ududff](4.84,-0.84)
\psdots[dotstyle=*,linecolor=xdxdff](3.,0.)
\psdots[dotstyle=*,linecolor=xdxdff](7.02,0.)
\psdots[dotstyle=*,linecolor=ududff](7.,1.)
\psdots[dotstyle=*,linecolor=ududff](6.24,0.76)
\psdots[dotstyle=*,linecolor=xdxdff](6.,0.)
\psdots[dotstyle=*,linecolor=ududff](6.28,-0.82)
\psdots[dotstyle=*,linecolor=ududff](7.02,-0.94)
\psdots[dotsize=4pt 0,dotstyle=*,linecolor=darkgray](0.,0.)
\end{scriptsize}
\end{pspicture*}
\end{center}
\caption{A graph $G$ and a splitting graph $G'$ of $G$. }
\label{depthnottrue}
\end{figure}
\end{Example}

In general, if $G$ is a chordal graph, the splitting graph of $G$ may not be again chordal. However, the following two results show that the splitting graph of a graph $G$ remains in the same family, when $G$ is a bipartite graph or a tree.

An  independent set of $G$ is a subset $W\subseteq V(G)$ such that $\{i,j\}\nsubseteq  W$ for all edges $\{i,j\}$ of $G$.

\begin{Proposition}
\label{stablebi}
If $G$ is a bipartite graph, then any splitting graph of $G$ is so.
 \end{Proposition}

\begin{proof}
Let $G$ be a bipartite graph with the vertex partition $X\cup Y$, where $X$ and $Y$ are independent sets of $G$. Consider any splitting graph $G'$ of $G$ with the surjective map $\alpha\: V(G')\to V(G)$ attached to it. Set $X'=\{x\in V(G'):\ \alpha(x)\in X\}$ and
$Y'=\{x\in V(G'):\ \alpha(x)\in Y\}$. Then $X'\cup Y'$ is a partition of $V(G')$. For any two vertices $x_1,x_2\in X'$, we have $\{\alpha(x_1),\alpha(x_2)\}\notin E(G)$, since $X$ is an independent set of $G$. Thus $\{x_1,x_2\}\notin E(G')$. Hence $X'$ is an independent set of $G'$. Similarly $Y'$ is an independent set of $G'$. Thus $G'$ is bipartite with the vertex partition $X'\cup Y'$.

\end{proof}

\begin{Proposition}
\label{forest2}
 Let $G$ be a forest. Then any splitting graph of $G$ is a forest.
 \end{Proposition}

\begin{proof}
Suppose that $G'$ is not a forest and $e'_1,\ldots,e'_m$ be a closed walk in $G'$, where $e'_1,\ldots,e'_m$ are pairwise distinct. Let $\alpha\: V(G')\to V(G)$ be the map attached to the splitting graph of $G$. Then $\alpha(e'_1),\ldots,\alpha(e'_m)$ is a closed walk in $G$ with pairwise distinct edges, a contradiction.
\end{proof}

 In shifting theory, in particular for symmetric algebraic shifting, one uses the so-called   {\em stretching operator}, see \cite{HH} and \cite{Ka}. Let $K$ be a field and $\tilde{S}=K[x_1,x_2,\ldots]$ be the  polynomial ring in infinitely many variables, and let $\Mc$ be the set of monomials of $\tilde{S}$. The {\em stretching operator} is the  map $\sigma\: \Mc\to \Mc$ which assigns to a monomial $u=x_{i_1}x_{i_2}\cdots x_{i_d}$ with $i_1\leq i_2\leq  \cdots \leq i_d$ the stretched monomial $\sigma(u)=x_{i_1}x_{i_2+1}x_{i_3+2}\cdots x_{i_d+(d-1)}$.  It is clear that an iterated application of $\sigma$ transforms $u$ into a squarefree monomial ideal.

 Now let $S=K[x_1,\ldots,x_n]$,  $I\subset S$ a monomial ideal and $G(I)=\{u_1, \ldots, u_m\}$ be the unique minimal monomial set of generators of $I$. Then in a suitable  polynomial ring $S'=K[x_1,\ldots,x_r]$ with $r\geq n$ one has  $\{\sigma(u_1), \ldots, \sigma(u_m)\}\subset S'$, and we let  $I^\sigma$ be the ideal in $S'$ generated by the monomials  $\sigma(u_1),\ldots,\sigma(u_m)$. Usually we assume that $S'$ is the polynomial ring with $r$ chosen minimal such that the monomials $\sigma(u_i)$ belong to it.
 The following examples illustrate again its effect.

 Let  $I=(x_1x_3x_5,x_1^2x_4^3x_7)\subset K[x_1,\ldots,x_9]$, then
 \[I^\sigma=(x_1x_4x_7,x_1x_2x_6x_7x_8x_{12})\subset K[x_1,\ldots,x_{14}].\]

Applying  $\sigma$ $t$-times to  $u=x_{i_1}\cdots x_{i_d}$ with $i_1 \leq  i_2\leq \cdots \leq i_d$   we get ${\sigma^t(u)}=x_{i_1}x_{i_2+t}x_{i_3+2t}\cdots x_{i_d+t(d-1)}$.
 We let
\[
I^{\sigma^t}=(\sigma^t(u_1),\ldots,\sigma^t(u_m)) \in S_t,
\]
where $S_t=K[x_1,\ldots, x_{n_t}]$, $n_{t}=n+t(d-1)$ and $d=\max\{\deg(u):\ u\in G(I)\}$.

For example, if $I=(x_1x_2,x_2x_3)$, then $I^\sigma=(x_1x_3,x_2x_4)$. In this example, $I$ has a linear resolution, while $I^\sigma$ does not.  Thus, unlike polarization, which preserves the graded Betti numbers of a monomial ideal, this is not the case for the operator $\sigma$, unless the monomial ideal is strongly stable, see for example \cite{HH} for a detailed discussion.

For any graph $G$ on the vertex set $[n]$, let $G^{\sigma}$ be a graph defined by the equation $I(G)^{\sigma}=I(G^{\sigma})$. Notice that $G^{\sigma}$ is a splitting graph of $G$. One can easily see that there exists a positive integer $t_0$ such that $G^{\sigma^{t}}\iso G^{\sigma^{t_0}}$ for all $t\geq t_0$. We denote $G^{\sigma^{t_0}}$ by $G^*$ and call it the $\sigma-$stable graph of $G$. Observe that $G^*$ depends on the labeling on the vertices of $G$. Indeed, consider the $4$-cycle $G$ with edges $\{1,2 \}, \{2,3\}, \{3,4\}$ and $\{1,4\}$. Then $G^*$ has the edges  $\{1,4 \}, \{2,5\}, \{3,6\},$ and $\{1,6\}$. Thus $G^*$ is a graph with $2$ connected components, where each of them is a path graph. On the other hand, if we relabel $G$ such that  $\{2,3 \}, \{2,4\}, \{1,4\}$ and $\{1,3\}$ are the edges of  $G$, then $G^*$ is again a $4$-cycle.

By the above observations, for any graph $G$, the $\sigma$-stable graph $G^*$ is a splitting graph of $G$. The splitting map for $G^*$ can be explicitly described. Namely, if  $E(G)=\{\{i_k,j_k\}\: k=1,\ldots,m\}$ with $i_k<j_k$ for all $k$, then $E(G^*)=\{\{i_k,j_k+t_0\}\: k=1,\ldots,m\}$ with $t_0$ big enough and the map $\alpha\: V(G^*)\to V(G)$ with $\alpha(i_k)=i_k$ and $\alpha(j_k+t_0)=j_k$ for $k=1,\ldots,m$ is surjective and  induces a bijection between the edges of $G^*$ and $G$.

Not all splitting graphs of $G$ are of the form $G^*$ for a suitable  labeling of $G$, see Figure~\ref{three}.

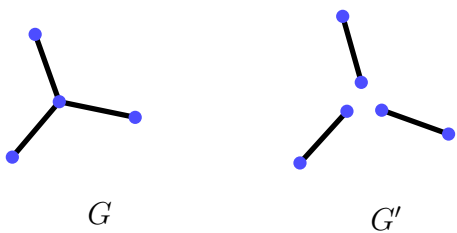
\begin{figure}[hbt]
\begin{center}
\newrgbcolor{ududff}{0.30196078431372547 0.30196078431372547 1.}
\psset{xunit=0.8cm,yunit=0.8cm,algebraic=true,dimen=middle,dotstyle=o,dotsize=5pt 0,linewidth=1.6pt,arrowsize=3pt 2,arrowinset=0.25}
\begin{pspicture*}(5.,0.)(16.,5.)
\psline[linewidth=2.pt](7.38,4.04)(7.78,2.92)
\psline[linewidth=2.pt](7.,2.)(7.78,2.92)
\psline[linewidth=2.pt](7.78,2.92)(9.04,2.66)
\rput[tl](8.25,1.24){$\large{G}$}
\psline[linewidth=2.pt](12.56,2.76)(11.78,1.9)
\psline[linewidth=2.pt](13.14,2.78)(14.25,2.38)
\psline[linewidth=2.pt](12.8,3.24)(12.49,4.34)
\rput[tl](12.95,1.18){$\large{G'}$}
\begin{scriptsize}
\psdots[dotstyle=*,linecolor=ududff](7.38,4.04)
\psdots[dotstyle=*,linecolor=ududff](7.,2.)
\psdots[dotstyle=*,linecolor=ududff](7.78,2.92)
\psdots[dotstyle=*,linecolor=ududff](9.04,2.66)
\psdots[dotstyle=*,linecolor=ududff](12.8,3.24)
\psdots[dotstyle=*,linecolor=ududff](12.56,2.76)
\psdots[dotstyle=*,linecolor=ududff](11.78,1.9)
\psdots[dotstyle=*,linecolor=ududff](13.14,2.78)
\psdots[dotstyle=*,linecolor=ududff](14.25,2.38)
\psdots[dotstyle=*,linecolor=ududff](12.49,4.34)
\end{scriptsize}
\end{pspicture*}
\end{center}
\caption{A splitting graph $G'$ of $G$ which is different from any $\sigma$-stable graph $G^*$ of $G$. }
\label{three}
\end{figure}

Note that if $G$ is connected with $n$ edges, then for each number $j\leq n$, there exists a splitting graph $G'$ of $G$ with $j$ connected components. However, this is not the case when we consider the set of $\sigma$-stable graphs $G^*$ of $G$. Therefore the  following question arises:
Let $G$ be a graph. For a given labeling $L$ of $G$, denote the number of connected components of the corresponding $G^*$ by $\gamma(L)$. Determine the set $C(G)=\{\gamma(L):\ L\ \textrm{is a labeling on }\ G\}$. For example if $G=P_n$, then $C(G)=[n-1]$ and if $G=C_n$, then $1\in C(G)$ if and only if $n$ is even.

\end{document}